\documentclass[leqno]{amsart}
\usepackage{txfonts}
\usepackage[T1]{fontenc}
\usepackage[latin1]{inputenc}
\usepackage{graphicx}
\usepackage{mathrsfs}
\usepackage[english]{babel}
\usepackage{amssymb,hyperref,amsmath}
\theoremstyle{plain}
\newtheorem{theo}{Theorem}[section]

\newtheorem{lemm}[theo]{Lemma}

\theoremstyle{definition}
\newtheorem{defi}[theo]{Definition}

\DeclareMathAlphabet{\mathrmsl}{OT1}{cmr}{m}{sl}
\renewcommand{\leq}{\leqslant}
\renewcommand{\geq}{\geqslant}

\newcommand{\RM}{{\mathbb R}}

\DeclareMathOperator{\Ric}{Ric}

\newcommand{\vol}{\operatorname{vol}}
\newcommand{\Scal}{\operatorname{Scal}}
\newcommand{\Hess}{\operatorname{Hess}}
\newcommand{\tr}{\operatorname{tr}}

\newcommand{\proofof}[1]{\end{#1}\begin{proof}}

\newcounter{mnotecount}[section]
\renewcommand{\themnotecount}{\thesection.\arabic{mnotecount}}
\newcommand{\mnote}[1]
{\protect{\stepcounter{mnotecount}}$^{\mbox{\footnotesize  $
      \bullet$\themnotecount}}$ \marginpar{\raggedright\tiny\em
    $\!\!\!\!\!\!\,\bullet$\themnotecount: #1} }

\swapnumbers

\makeatletter
\@addtoreset{equation}{section}
\makeatother

\renewcommand{\footnote}[1]{${}^($\footnotemark${}^)$ \footnotetext{#1}}

\begin{document}
\title[Asymptotic invariants and the Ricci tensor]{Computing asymptotic invariants with the \\ Ricci tensor
on asymptotically flat and\\ asymptotically hyperbolic manifolds}
\author{Marc Herzlich} 
\address{Universit\'e de Montpellier\\ Institut montpelli\'erain Alexander Grothendieck\\ 
UMR 5149 CNRS -- UM\\ Montpellier\\ France} 
\email{marc.herzlich@umontpellier.fr}
\date{\today}
\thanks{The author is supported in part by the project ANR-12-BS01-004 `Geometry and topology of open manifolds' of the French National Agency for Research.}
\begin{abstract} 
We prove in a simple and coordinate-free way the equivalence between the classical definitions of the mass or of the center
of mass of an asymptotically flat manifold and their alternative definitions depending on the Ricci tensor and conformal
Killing fields. This enables us to prove an analogous statement in the asymptotically hyperbolic case.
\end{abstract}

\maketitle

\section*{Introduction}

Mass is the most fundamental invariant of asymptotically flat manifolds. Originally defined in
General Relativity, it has since played an important role in Riemannian geometric issues. 
Other interesting invariants, still motivated by physics, include the energy momentum, the angular momentum, 
and the center of mass (which will be of interest in this note).
Moreover, they have been extended to other types of asymptotic behaviours such as asymptotically hyperbolic manifolds. 

Two difficuties occur when handling the mass of an asymptotically flat or hyperbolic manifold (or any of its companion
invariants): it is defined as a limit of an integral expression on larger and larger spheres, and
it depends on the first derivatives of the metric tensor written in a special chart where the metric coefficients are 
asymptotic to those of the model (flat, hyperbolic) metric at infinity.

It seems unavoidable that a limiting process is involved in the definitions. But
finding expressions that do not depend on the first derivatives but on rather more geometric quantities is an old 
question that has attracted the attention of many authors. It was suggested by A. Ashtekhar and R. O. Hansen 
\cite{ashtekar-hansen_unified} (see also P. Chru\'sciel \cite{chrusciel_rk-positive-energy}) 
that the mass could be rather defined from the Ricci tensor and a conformal Killing field of 
the Euclidean space. Equality between the two definitions, as well as a similar identity for the center of mass, has 
then been proved rigorously by L.-H. Huang using a density theorem \cite{huang_conf-chinoise}, \emph{cf}.
previous work by J. Corvino and H. Wu \cite{corvino-wu} for conformally flat manifolds, 
and by P. Miao and L.-F.~Tam \cite{miao-tam} through a direct computation in coordinates.

The goal of this short note is twofold: we shall provide first a simple proof of the equality between the classical definitions of the asymptotic
invariants and their alternative definitions using the Ricci tensor. Although similar in spirit to Miao-Tam \cite{miao-tam}, our approach completely avoids computations in
coordinates. Moreover, it clearly explains why the equality should hold, by connecting it to a natural
integration by parts formula related to the contracted Bianchi identity. A nice corollary of our proof is that it can be naturally extended to other settings where asymptotic invariants have been defined. As an example of this feature, we  
provide an analogue of our results in the asymptotically hyperbolic setting.

\smallskip

\section{Basic facts}

We begin by recalling the classical definitions of the mass and the center of mass of an asymptotically flat 
manifold, together with their alternative definitions involving the Ricci tensor. In all that follows, the dimension $n$ 
of the manifolds considered will be taken to be at least $3$. We shall restrict ourselves to manifolds with only one end, but the definitions can be straightforwardly extended to the general case.

\begin{defi}\label{defn1.1} 
An asymptotically flat manifold is a complete Riemannian manifold $(M,g)$ such that there exists a diffeomorphism $\Phi$ 
(called a chart at infinity) from the complement of a compact set in $M$ into the complement of a ball in $\RM^n$, 
such that, in these coordinates and for some $\tau>0$,
$$ |g_{ij} - \delta_{ij} | = O(r^{-\tau}), \quad |\partial_kg_{ij}| = O(r^{-\tau-1}), 
\quad |\partial_k\partial_{\ell}g_{ij}| = O(r^{-\tau-2}), $$
where $r=|x|$ is the Euclidean radius in $\RM^n$.
\end{defi}

\begin{defi}\label{defn1.2} 
If $\tau>\tfrac{n-2}{2}$ and the scalar curvature of $g$ is integrable, the quantity
\begin{equation}\label{eq_defn1.2}
m(g) = \frac{1}{2(n-1)\omega_{n-1}}\, \lim_{r\to\infty} \int_{S_r} (-\delta^e g - d \tr_e g)(\nu^e)\,d\!\vol^e_{s_r} 
\end{equation}
exists (where $e$ refers to the Euclidean metric in the given chart at infinity, $\delta$ is the divergence defined as the
adjoint of the exterior derivative, $\nu^e$ denotes the field of Euclidean outer unit normals to the coordinate spheres $S_r$,
and $\omega_{n-1}$ is the volume of the unit round sphere of
dimension $n-1$) and is independent of the chart chosen around infinity. It is called the \emph{mass} 
of the asymptotically flat manifold $(M,g)$.
\end{defi}

\begin{defi}\label{defn1.3} 
If $\tau>\tfrac{n-2}{2}$, the scalar curvature $\Scal^g$ of $g$ is integrable, $m(g)\neq 0$, and the following so-called
\emph{Regge-Teitelboim (RT) conditions} are satisfied:
$$ |g_{ij}^{\textrm{odd}} | = O(r^{-\tau-1}), \quad |\partial_k\left(g_{ij}^{\textrm{odd}}\right)| = O(r^{-\tau-2}), \quad
\left(\Scal^g\right)^{\textrm{odd}} = O(r^{-2\tau-2}) $$ 
(where $\cdot^{\textrm{odd}}$ denotes the odd part of a function on the chart at infinity),
the quantity
\begin{equation*}
\begin{split}
c^{\alpha}(g) \ = \ & \frac{1}{2(n-1)\omega_{n-1}m(g)}\, \lim_{r\to\infty} \int_{S_r} 
\bigl[ \, x^{\alpha}(-\delta^e g - d \tr_e g) \\ 
& \ \ \ \ \ \ \ \ \ \ \ \ \ \ \ \ \ \ \ \ \ \ \ \ \ \ \ \ \ \ \ \ \ \ \ \ \ \ \ \ \ 
- (g-e)(\partial_\alpha,\cdot) + \tr_e (g-e)\, dx^{\alpha}\,\bigr] (\nu)\,d\!\vol^e_{s_r} 
\end{split}
\end{equation*}
exists for each $\alpha$ in $\{1,...,n\}$. Moreover, the vector $\mathbf{C}(g) = (c^1(g),\dots,c^n(g))$ is independent 
of the chart chosen around infinity, up to the action of rigid Euclidean isometries. It is called the \emph{center of 
mass} of the asymptotically flat manifold $(M,g)$.
\end{defi}

The normalization factors may seem somewhat arbitrary in the previous two definitions: they however show up naturally if one wants these invariants to be equal to the usual parameters of the standard spacelike slices of the Schwarzschild metrics (in any dimension).

Existence and invariance of the mass have been proved by R. Bartnik \cite{ba} and P. T. Chru{\'s}ciel \cite{chrusciel-mass}. 
The center of mass has been introduced by T. Regge and C. Teitelboim \cite{regge-teitelboim_surface-integrals,regge-teitelboim_improved},
and R. Beig and N. \'O Murchadha \cite{bom}, see also the more recent works of J. Corvino and R. Schoen 
\cite{corvino_scalar-curvature,corvino-schoen}. 

We shall recall here the approach towards existence and well-definedness of these invariants due to B. Michel 
\cite{michel_geometric-invariance}. Let $g$ and $b$ be two metrics on a complete manifold $M$, the latter one being considered as a 
\emph{background} metric, hence the notation. Let also $\mathcal{F}^g$ (resp. $\mathcal{F}^b$) be a (scalar) polynomial invariant in the curvature 
tensor and its subsequent derivatives, $V$ be a function, and $(M_r)_{r\geq 0}$ be an exhaustion of $M$ by compact subsets, whose boundaries will 
be denoted by $S_r$ (later taken as large coordinate spheres in a chart at infinity). One then may compute:
\begin{equation*}
\int_{M_r} V\, \left(\mathcal{F}^g - \mathcal{F}^b\right)  \, d\!\vol^b 
\ = \ \int_{M_r} V\, (D\mathcal{F})_b(g-b)\, d\!\vol^b \ + \ \int_{M_r} V \, \mathcal{Q}(b,g) \, d\!\vol^b 
\end{equation*}
where $\mathcal{Q}$ denotes the (quadratic) remainder term in the Taylor formula for the functional $\mathcal{F}$.
Integrating the linear term by parts leads to:
\begin{equation}\label{eqn.michel}
\begin{split}
\int_{M_r} \!V\left(\mathcal{F}^g - \mathcal{F}^b\right) d\!\vol^b  
\, = & \, \int_{M_r} \!\langle (D\mathcal{F})_b^*V\, ,\, g-b\rangle \, d\!\vol^b \, + \, \int_{S_r} \!\mathbb{U}(V,g,b) \\
& \ \ \ \ \ \ \ \ \ + \, \int_{M_r} \!\!V\,\mathcal{Q}(b,g) \, d\!\vol^b  
\end{split}
\end{equation}
(where we include here the volume element in the definition of $\mathbb{U}$).
This formula shows that 
$$\mathcal{H}_{\mathcal{F}}(V,g,b) = \lim\limits_{r\to\infty} \int_{S_r} \mathbb{U}(V,g,b)$$ exists if the following two natural conditions are satisfied:
\begin{enumerate}
\item[(1)] the metric $g$ is asymptotic to $b$ so that $V\, \left(\mathcal{F}^g - \mathcal{F}^b\right)$ and $V\, \mathcal{Q}(b,g)$ are integrable;
\item[(2)] $V$ belongs to the kernel of $(D\mathcal{F})_b^*$ (the adjoint of the 
first variation operator of the Riemannian functional $\mathcal{F}$). 
\end{enumerate}
Moreover, Michel proves in \cite{michel_geometric-invariance}
that $\mathcal{H}_{\mathcal{F}}(V,g,b)$ is an asymptotic invariant, independent of the choice of chart at infinity, if
\begin{enumerate}
\item[(3)] the background geometry $b$ is rigid enough, in the sense that any two `charts at infinity' where $g$ is asymptotic to $b$ differ by a
diffeomorphism whose leading term is an isometry of $b$;
\item[(4)] $\mathcal{F}^b$ is a constant function.
\end{enumerate}
This last result is a consequence of the diffeomorphism invariance 
of the integrated scalar invariant $\mathcal{F}^g$.
 
If one chooses $\mathcal{F}^g = \Scal^g$ on an asymptotically flat manifold (hence $b=e$, the Euclidean metric), one has
$$ (D\Scal)_e^*V \, = \, \Hess^e V \, + \, (\Delta^eV)\,e, $$ 
where $\Hess^e$ denotes the Hessian of a function and $\Delta^e$ is the Euclidean Laplace operator (defined here as the opposite of the trace of the Hessian, so that it has non-negative spectrum).
Its kernel then consists of affine functions. We now let $V\equiv 1$. 
The scalar curvature of $g$ is integrable and $\tau>\tfrac{n-2}{2}$ in
Definition \ref{defn1.1}, hence
\begin{equation}\label{eqn1.3bis}
\mathcal{H}_{\Scal}(1,g,e) \ = \ \lim_{r\to\infty} \int_{M_r} \Scal^g \, d\!\vol^e \ 
- \ \lim_{r\to\infty} \int_{M_r}  \mathcal{Q}(e,g) \, d\!\vol^e  
\end{equation}
makes sense since integrability of $\Scal^g$ yields convergence of the first term, whereas the integrand in the second term is a combination of terms in $(g-b)\partial^2g$ and $g^{-1}(\partial g)^2$ which are integrable due to the value of $\tau$.
Moreover, an easy computation shows that
$$ \mathcal{H}_{\Scal}(1,g,e) \ = \ \lim_{r\to\infty} \int_{S_r} \mathbb{U}(1,g,e) \ = \ 2(n-1)\omega_{n-1}\, m(g)  $$
where $m(g)$ is the classical definition of the mass given in Definition \ref{defn1.2}.
Moreover, Michel's analysis recalled above \cite{michel_geometric-invariance} shows that the mass is an asymptotic invariant, independent of the choice of chart at infinity, since Euclidean geometry is a rigid background geometry \cite{ba,chrusciel-mass} and $\Scal^e \equiv 0$.

If one takes $V=V^{(\alpha)}=x^{\alpha}$ (the $\alpha$-th coordinate function in the chart at infinity, for any 
$\alpha$ in $\{1,...,n\}$), the same procedure now yields the classical definition of the $\alpha$-th coordinate of the center of mass, \emph{i.e.}
$$ \mathcal{H}_{\Scal}(V^{(\alpha)},g,e) \ = \ \lim_{r\to\infty} \int_{S_r} \mathbb{U}(V^{(\alpha)},g,e) \ = \ 2(n-1)\omega_{n-1}\, m(g)\, c^{\alpha}(g) $$
for any $\alpha \in \{1,...,n\}$. 
Under the RT conditions, these converge as well and the vector $\mathbf{C}(g)$ is again an 
asymptotic invariant.

We now recall the alternative definitions of these asymptotic invariants \emph{via} the Ricci tensor, following the suggestions of A. Ashtekhar  and R. O Hansen, P.~Chru\'sciel, etc.:

\begin{defi}\label{defn1.4} 
Let $X$ be the radial vector field $X = r\partial_r$ in the chosen chart at infinity. Then we define 
\emph{the Ricci version of the mass} of $(M,g)$ by
\begin{equation}\label{eqn1.4}
m_R(g) \ = \ - \frac{1}{(n-1)(n-2)\omega_{n-1}}\, 
\lim_{r\to\infty} \int_{S_r} \left(\Ric^g - \frac12\Scal^g g\right)(X,\nu)\, d\!\vol^g
\end{equation}
whenever this limit is convergent. For $\alpha$ in $\{1,\dots,n\}$, let $X^{(\alpha)}$ be the Euclidean conformal Killing field 
$X^{(\alpha)} = r^2\partial_{\alpha} - 2 x^{\alpha}x^i\partial_i$ and define
\emph{the Ricci version of the center of mass}:
\begin{equation}\label{eqn1.4bis}
c^{\alpha}_R(g) \ = \ \frac{1}{2(n-1)(n-2)\omega_{n-1}m(g)}\, 
\lim_{r\to\infty} \int_{S_r} \left(\Ric^g - \frac12\Scal^g g\right)(X^{(\alpha)},\nu)\, d\!\vol^g
\end{equation}
whenever this limit is convergent. We will call this vector $\mathbf{C}_R(g)=(c^{1}_R(g),\dots,c^n_R(g))$.
\end{defi}

Notice that these definitions of the asymptotic invariants rely on the Einstein tensor,
which seems to be consistent with the physical motivation.

\section{Equality in the asymptotically flat case}

In this section, we will prove the equality between the classical expressions $m(g)$ and $\mathbf{C}(g)$ of the mass and the
center of mass and their Ricci versions $m_R(g)$ and $\mathbf{C}_R(g)$. The proof we will give relies on Michel's approach
described above together with two elementary computations in Riemannian geometry.

\begin{lemm}[The integrated Bianchi identity]\label{lemma2.1}
Let $h$ be a $C^3$ Riemannian metric on a smooth compact domain with boundary $\Omega$ and $X$ be a conformal Killing field.
Then 
$$ \int_{\partial \Omega} \left(\Ric^h - \frac12\Scal^h h\right)(X,\nu)\, d\!\vol^h_{\partial \Omega} 
\ =  \ \frac{n-2}{2n} \int_{\Omega} \Scal^h \left(\delta^h X\right)\, d\!\vol^h_{\Omega} \  ,$$
where $\nu$ is the outer unit normal to $\partial \Omega$.
\end{lemm}

\begin{proof} This equality is a variation of the well known \emph{Pohozaev identity} in conformal geometry, as
stated by R. Schoen \cite{schoen_ricci-mass}. Our version has the advantage that the divergence of $X$ appears in
the bulk integral (the classical Pohozaev identity is rather concerned with the derivative of the scalar curvature
in the direction of $X$).The proof being very simple, we will 
give it here. From the contracted Bianchi identity $\delta^h\left(\Ric^h - \tfrac12\Scal^h h\right)=0$, 
one deduces that
$$
\int_{\partial \Omega} \left(\Ric^h - \frac12\Scal^h h\right)(X,\nu)\, d\!\vol^h_{\partial \Omega} \ = \  
\int_{\Omega} \langle \Ric^h - \frac12\Scal^h h,(\delta^h)^*X\rangle_h \, d\!\vol^h_{\Omega} 
$$
where $(\delta^h)^*$ in the above computation denotes the adjoint of the divergence on vectors, \emph{i.e.} the symmetric 
part of the covariant derivative. Since $X$ is conformal Killing, $(\delta^h)^*X = -\tfrac1n(\delta^hX)h$ and
\begin{align*} \int_{\partial \Omega} \left(\Ric^h - \frac12\Scal^h h\right)(X,\nu)\, d\!\vol^h_{\partial \Omega} 
& \ = \ - \frac1n \int_{\Omega} \tr_h\left( \Ric^h - \frac12\Scal^h h\right)(\delta^h X) \, d\!\vol^h_{\Omega} \\
& \ = \ \frac{n-2}{2n} \int_{\Omega} \Scal^h (\delta^hX) \, d\!\vol^h_{\Omega}
\end{align*} 
and this concludes the proof. \end{proof}

Lemma \ref{lemma2.1} provides a link between the integral expression appearing in the Ricci definition
of the asymptotic invariants (see \eqref{eqn1.4}) and the bulk integral 
$$\int_{\Omega} \Scal^h \left(\delta^h X\right)\, d\!\vol^h_{\Omega}\, . $$ 
This latter
quantity also looks like the one used by Michel to derive the definitions of the asymptotic invariants, provided that
some connection can be made between divergences of conformal Killing fields and elements in the kernel of the adjoint of
the linearized scalar curvature operator. Such a connection stems from our second lemma:

\begin{lemm}\label{lemma2.2}
Let $h$ be a $C^3$ Riemannian metric and $X$ a conformal Killing field. If $h$ is Einstein with Einstein constant 
$\lambda(n-1)$, then $V=\delta^hX$ sits in the kernel of $(D\Scal)_h^*$. More precisely:
\begin{equation}
\label{eqlemm2.2}
\Hess^h V \ = \ - \lambda\, V\, h .
\end{equation}
\end{lemm}

\begin{proof} Recall that $(D\Scal)_h^* V = \Hess^h V + (\Delta^hV)h - V\Ric^h$ \cite[1.159(e)]{besse}, so that 
its kernel is precisely the set of solutions of \eqref{eqlemm2.2}  if $\Ric^h = \lambda(n-1)h$.

Let $\phi_t$ be the (local) flow of $X$, which acts by conformal diffeomorphisms, and $e^{2u_t}$ the conformal
factor at time $t\geq 0$, with $u_0=0$. Hence 
$$\Ric^{\phi_t^*h} = \lambda(n-1)\,\phi_t^*h,$$
an equality which can be written equivalently as 
$$\Ric^{e^{2u_t}h} = \lambda(n-1)\,e^{2u_t}h$$
since $\phi_t$ is conformal. From \cite[1.159(d)]{besse},
$$
Ric^{e^{2u_t}h} \ = \ Ric^{h} - (n-2)\left(\Hess^h u_t - du_t\otimes du_t\right) 
                          \ + \ \left(\Delta^h u_t - (n-2)\,|du_t|^2_h\right)\, h ,
$$
from which one deduces that 
$$  - (n-2)\left(\Hess^h u_t - du_t\otimes du_t\right) \ + \ \left(\Delta^h u_t - (n-2)|du_t|^2_h\right)\, h 
\ = \ \lambda(n-1)\,\left(e^{2u_t} - 1\right) h . $$
We now differentiate at $t=0$. Denoting by $\dot{u}$ the first variation of $u_t$, the conformal Killing equation yields the following relation between $X$ and $\dot{u}$:
$\delta^hX = -n\,\dot{u}$. Taking into account that $u_0=0$, one gets:
\begin{equation}\label{eqnpreuvelemm2.2} 
- (n-2)\Hess^h \dot{u} + (\Delta^h \dot{u})\, h  \ = \ 2(n-1)\lambda\,\dot{u}\, h
\end{equation}
(note that this equation can also be obtained directly from the conformal Killing equation
on $X$ but we prefer the proof above as it underlines the relation with the variations of
the Ricci curvature under conformal deformations).
Tracing this identity yields $2(n-1)\,\Delta^h \dot{u} = 2n(n-1)\,\lambda\,\dot{u}$, so that 
$\Delta^h \dot{u} = n\lambda\,\dot{u}$. Inserting this in Equation~\eqref{eqnpreuvelemm2.2} leads to
$\Hess^h \dot{u} = - \lambda\,\dot{u}\, h$, which is the desired expression.
\end{proof}

We now have all the necessary elements to prove the equality between the classical expressions of the 
asymptotic invariants and their Ricci versions in the asymptotically flat case. 

\begin{theo}
If $(M,g)$ is a $C^3$ asymptotically flat manifold with integrable scalar curvature and decay rate $\tau > \tfrac{n-2}{2}$, 
then the classical and Ricci definitions of the mass agree: $m (g) \, = \, m_R(g)$.
If $m(g)\neq 0$ and the RT asymptotic conditions are moreover assumed, the same holds for the center of mass, \emph{i.e.}
$c^{\alpha}(g) \, = \, c^{\alpha}_R(g)$ for any $\alpha \in \{1,...,n\}$.
\end{theo}

\begin{proof} 
We shall give the complete proof for the mass only, the case of the center of mass being entirely similar.
Fix a chart at infinity on $M$. As the mass is defined asymptotically, we may freely replace a compact part in $M$
by a (topological) ball, which we shall decide to be the unit ball $B_0(1)$ in the chart at infinity. The manifold is 
unchanged outside that compact region. For any $R>>1$ we define a cut-off function $\chi_R$ which vanishes inside the
sphere of radius $\tfrac{R}{2}$, equals $1$ outside the sphere of radius $\tfrac{3R}{4}$ and moreover satisfies
$$ |\nabla\chi_R| \leq C_1 R^{-1}, \quad |\nabla^2\chi_R| \leq C_2 R^{-2}, \ \ \textrm{and } \,
|\nabla^3\chi_R| \leq C_3 R^{-3}$$
for some universal constants $C_i$ ($i=1,2,3$) not depending on $R$. We shall now denote $\chi=\chi_R$ unless some
confusion is about to occur. We then define a metric on the annulus $\Omega_R = A(\tfrac{R}{4},R)$:
$$h \ =  \ \chi g \, + \, (1-\chi) e , $$
and we shall also denote by $h$ the complete metric obtained by gluing the Euclidean metric inside the ball
$B_0(\tfrac{R}{4})$ and the original metric $g$ outside the ball $B_0(R)$.

Let now $X$ be a conformal Killing field for the Euclidean metric. Lemma~\ref{lemma2.2} tells us tha $V =  \delta^e X$ sits in the kernel of the adjoint of the linearized scalar curvature operator, \emph{i.e.} $(D\Scal)_e^*V=0$. We now compute as in Lemma \ref{lemma2.1} over the annulus $\Omega_R = A(\tfrac{R}{4},R)$:
\begin{equation*}
\int_{S_R} \left(\Ric^h - \frac12\Scal^h h\right)(X,\nu^h)
\ = \ \int_{\Omega_R} \langle \Ric^h - \frac12\Scal^h h\ , (\delta^h)^*X\rangle \ ,
\end{equation*}
where the volume forms and scalar products are all relative to $h$ but have been removed for clarity. (Notice that the boundary contribution at $\tfrac{R}{4}$ vanishes since $h$ is flat there).
We now split $(\delta^h)^*X$ into its trace part $-\tfrac1n(\delta^hX)h$ and its tracefree part $(\delta^h)_0^*X$ (where $(\delta^h)_0$ is as above the conformal Killing operator), so that
\begin{equation*}
\begin{split}
\int_{S_R} \left(\Ric^h - \frac12\Scal^h h\right)(X,\nu^h) \, = \, & -\frac1n \int_{\Omega_R} \!\!\tr_h\left(\Ric^h - \frac12\Scal^h h\right)\delta^h X \\ 
& \ \ \ \ \ \ \ + \int_{\Omega_R} \!\!\langle \Ric^h - \frac12\Scal^h h\, , (\delta^h)_0^*X\rangle .
\end{split}
\end{equation*}
Hence,
\begin{equation}\label{eq2.1}
\begin{split}
\int_{S_R} \left(\Ric^h - \frac12\Scal^h h\right)(X,\nu^h)
\ = \ & \frac{n-2}{2n} \int_{\Omega_R} (\delta^h X) \,\Scal^h  \\
& \ \ \ \ \ \ \ + \, \int_{\Omega_R} \langle \Ric^h - \frac12\Scal^h h,(\delta^h)_0^*X\rangle \ . 
\end{split}
\end{equation}
 
We now choose $X=r\partial_r$ (the radial dilation vector field), so that $\delta^eX=-n$ in this case. 
We can now replace the volume form $d\!\vol^h$, 
the divergence $\delta^h$, and the conformal Killing operator $(\delta^h)^*_0$ by their Euclidean counterparts 
$d\!\vol^e$, $\delta^e$, and $(\delta^e)^*_0$: indeed, from our asymptotic decay conditions, our choice of cut-off function $\chi$, 
and the facts that $\tau>\tfrac{n-2}{2}$ and $|X|=r$, one has for the first term in the right-hand side 
of \eqref{eq2.1}: 
$$ \int_{\Omega_R} (\delta^h X) \,\Scal^h d\!\vol^h \, - \, \int_{\Omega_R} (\delta^e X) \,\Scal^h d\!\vol^e
\, = \, O\left(R^{n-2\tau-2}\right) \, = \, o(1)$$  
as $R$ tends to infinity (note that the second term in the left-hand side does not tend alone to zero at infinity as
the scalar curvature of $h$ may not be uniformly integrable).
As $(\delta^e)_0^*X=0$, the last term in \eqref{eq2.1} can be treated in the same way and it is $o(1)$, too. 
One concludes that, in the case $X$ is the radial field,
\begin{equation}\label{eq2.2}
\int_{S_R} \left(\Ric^h - \frac12\Scal^h h\right)(X,\nu^h)\, d\!\vol^e_{S_R} 
\ = \ \frac{n-2}{2n} \int_{\Omega_R} (\delta^e X) \,\Scal^h d\!\vol^e
\, + \ o(1).
\end{equation}
We now argue as in Michel's analysis (see Equation \eqref{eqn.michel}) but over the annulus $\Omega_R$:
\begin{equation*}
\int_{\Omega_R} \!(\delta^e X)\Scal^h d\!\vol^e \, = \, \int_{S_R} \!\!\mathbb{U}(\delta^eX,g,e)\, + \int_{S_{\frac{R}{4}}} \!\!\!\mathbb{U}(\delta^eX,h,e) 
\, + \int_{\Omega_R} \!\!(\delta^eX)\mathcal{Q}(e,h)\, d\!\vol^e .
\end{equation*}
But the boundary contribution at $r=\tfrac{R}{4}$ vanishes since $h=e$ there, and moreover, taking into account $\delta^eX=-n$, the assumptions on $\chi$, our asymptotic decay conditions, and $\tau>\tfrac{n-2}{2}$, the integral containing the $\mathcal{Q}$-term tends to $0$ when $R$ goes to infinity, for the very same reason that made it integrable in Michel's analysis. Thus, one gets
eventually:  
$$ \frac{1}{2(n-1)\omega_{n-1}}\,\int_{S_R} \left(\Ric^h - \frac12\Scal^h h\right)(r\partial_r,\nu^h)\, d\!\vol_{S_r}
\ = \ \frac{2-n}{2}\, m(g) \, + \ o(1),$$
and this proves the expected result.

If one now chooses $X=X^{(\alpha)} = r^2\partial_{\alpha} - 2 x^{\alpha}x^i\partial_i$, 
\emph{i.e.} $X$ is the essential conformal Killing field of $\mathbb{R}^n$ obtained by conjugating a translation by 
the inversion map, one has $\delta^ eX^{(\alpha)} = 2n x^{\alpha} = 2n V^{(\alpha)}$ and one can use the same argument.
Some careful bookkeeping shows that all appropriate terms are $o(1)$ due to the Regge-Teitelboim conditions and
one concludes that
$$ \frac{1}{2(n-1)\omega_{n-1}}\,\int_{S_r} \left(\Ric^h - \frac12\Scal^h h\right)(X^{(\alpha)},\nu^h) \, d\!\vol_{S_r}
\ = \ (n-2) \, m(g)\, c^{\alpha}(g) \, + \, o(1)$$
as expected.
\end{proof} 

\section{Asymptotically hyperbolic manifolds}

We now show that the same approach can be used to get analogous expressions in other
settings where asymptotic invariants have been defined. Looking back at
what has been done in the previous sections, we see that the proofs relied on the following two crucial facts: \begin{enumerate}
\item[(1)] the definition of the invariant should come (through Michel's analysis) from a Riemannian functional, which should in turn be related with some version of the Bianchi identity;
\item[(2)] there should exist some link between conformal Killing vectors and 
functions in the kernel of the adjoint linearized operator of the relevant Riemannian functional.
\end{enumerate}
In the presence of these two features, the equality between the classical definition of the invariants (\emph{\`a la} Michel) and their Ricci versions follows almost immediately, as the estimates necessary to cancel out all irrelevant terms are exactly the same as those used in the definition of the invariants, see for instance Equation \eqref{eq2.2} and the arguments surrounding it.

We insist on the fact that this idea is completely general and might be applied to a number of different geometric settings. As an example of this, we shall study the case of asymptotically hyperbolic
manifolds. The mass was defined there by P. T. Chru{\'s}ciel and the author 
\cite{ptc-mh} and independently by X. Wang \cite{wang_hyperbolic-mass}, see \cite{mh_survey-AH-mass} for a comparison.

\begin{defi}
An asymptotically hyperbolic manifold (with one end) is a complete Riemannian manifold $(M,g)$ such that there exists a diffeomorphism $\Phi$ 
(chart at infinity) from the complement of a compact set in $M$ into the complement of a ball in 
$\mathbb{H}^{n}$ (equipped with the background hyperbolic metric $b = dr^2 + \sinh^2 r g_{\mathbb{S}^{n-1}}$), satisfying the 
following condition: if $(\epsilon_0=\partial_r, \epsilon_1$, ..., $\epsilon_n)$ is a $b$-orthonormal basis, and 
$g_{ij} = g(\epsilon_{i},\epsilon_{j})$, there exists some $\tau>0$ such that,
$$ |g_{ij} - \delta_{ij} | = O(e^{-\tau r}), \quad |\epsilon_k\cdot g_{ij}| = O(e^{-\tau r}), 
\quad |\epsilon_k\cdot \epsilon_{\ell}\cdot g_{ij}| = O(e^{-\tau r}). $$
\end{defi}

\begin{defi}\label{defn3.2} 
If $\tau>\tfrac{n}{2}$ and $\left(\Scal^g + n(n-1)\right)$ is integrable in $L^1(e^rd\!\vol_b)$, 
the linear map $\mathbf{M}(g)$ defined by:
\begin{equation*}
\begin{split}
V \, \longmapsto \, & \frac{1}{2(n-1)\omega_{n-1}}\, \lim_{r\to\infty} \int_{S_r} 
\bigl[\, V \, (-\delta^b g - d \tr_b g) \\ 
&  \ \ \ \ \ \ \ \ \ \ \ \ \ \ \ \ \ \ \ \ \ \ \ \ \ \ \ \ \ \ \ \ \ \ + \tr_b(g-b) dV - (g-b)(\nabla^bV,\cdot)\, \bigr] (\nu)  \,d\!\vol_{s_r} 
\end{split}
\end{equation*}
is well-defined on the kernel of $(D\Scal)^*_b$ and is independent of the chart at infinity. It is called the 
\emph{mass of the asymptotically hyperbolic manifold} $(M,g)$.
\end{defi}

As in the asymptotically flat case, the normalization
factor comes from the computation for a reference family of metrics, which are the \emph{generalized Kottler
metrics} in the asymptotically hyperbolic case.

Existence and invariance of the mass can  be proven using Michel's approach \cite{michel_geometric-invariance}. 
The kernel $\mathcal{K} = \ker (D\Scal)^*_b$ is the space of functions $V$ solutions of $$\Hess^b V = V\, b.$$
It is $(n+1)$-dimensional and is generated, in the coordinates above, by the functions $V^{(0)}=\cosh r$, 
$V^{(\alpha)}=x^{\alpha}\sinh r$ (for $\alpha\in\{1,\dots,n\}$),
where $(x^{\alpha}) = (x^{1},\dots,x^n)$ are the Euclidean coordinates on the unit sphere induced by the standard embedding 
$\mathbb{S}^{n-1} \subset \mathbb{R}^n$. (An alternative definition of $\mathcal{K}$ is provided as follows: when the hyperbolic space is seen as the upper hyperboloid in Minkowski spacetime, then $\mathcal{K}$ is the set of functions generated by the 
restrictions of the Minkowskian coordinate functions to the hyperboloid.)

Contrarily to the asymptotically flat case, the center of mass is already included here and doesn't need to
be defined independently. Indeed, the space $\mathcal{K}$ is an irreducible representation of $O_0(n,1)$ (the isometry group
of the hyperbolic space), so that
all functions $V$ contribute to a single vector-valued invariant $\mathbf{M}(g)$. In the asymptotically flat case, 
the kernel $\mathcal{K}$ splits into a trivial $1$-dimensional representation 
(the constant functions) which gives rise to the mass, and the standard representation of $\mathbb{R}^n\rtimes O(n)$ on 
$\mathbb{R}^n$ (the linear functions), which gives birth to the (vector-valued) center of mass.
 
The hyperbolic conformal Killing fields are the same as those of the Euclidean space, but their divergences must now be 
explicited with respect to the hyperbolic metric. In the ball model of the hyperbolic space, one computes that 
$\delta^bX^{(0)} = -nV^{(0)}$ for the radial dilation vector field $X^{(0)}$, whereas $\delta^bX^{(\alpha)} = -nV^{(\alpha)}$
for the (inverted) translation fields. 

We can now argue as above, but starting with the modified Einstein tensor
$$\tilde{G}^g \ = \ \Ric^g \, - \, \frac12 \Scal^g g \, - \, \frac{(n-1)(n-2)}{2} g \ . $$
The Bianchi-like formula analogous to that of Lemma \ref{lemma2.1} reads, for any conformal
Killing field $X$,
$$
\int_{\partial \Omega} \tilde{G}^g(X,\nu)\, d\!\vol^g_{\partial \Omega} 
\ =  \ \frac{n-2}{2n} \int_{\Omega} \left(\Scal^g + \, n(n-1)\right)\,\delta^g X\, d\!\vol^g_{\Omega} \  , 
$$ 
and we note that the right-hand side is the expected expression to apply Michel's approach for the definition of the mass \cite{ptc-mh,michel_geometric-invariance}. 
The sequel of the proof is now completely similar to
the one given above. The very same arguments that provide convergence of the mass in Michel's approach show that all irrelevant contributions at infinity cancel out, so that, keeping the same notation as in the previous sections (the only difference being that polynomial decay estimates must be changed to exponential ones),
\begin{align*}
\int_{\Omega_R} (\delta^h X) \left(\Scal^h + n(n-1)\right) d\!\vol^h
& \ = \ \int_{\Omega_R} (\delta^b X) \left(\Scal^h + n(n-1)\right) d\!\vol^b \, + \, o(1)\\ 
& \ = \ \int_{S_R} \mathbb{U}(\delta^bX,g,b) 
\, + \, o(1) .
\end{align*}
The relation between the divergences of the conformal Killing vectors and the elements in the kernel of the adjoint linearized operator 
comes again from Lemma \ref{lemma2.2}, and one concludes as above with the following alternative definition of the mass 
involving the Ricci tensor:

\begin{theo}
For any $i\in\{0,...,n\}$, 
$$  
\ \mathbf{M}(g)\left[V^{(i)}\right] \ 
= - \frac{1}{n} \,\mathbf{M}(g)\left[\delta^bX^{(i)}\right] \ 
= \ - \frac{1}{(n-1)(n-2)\omega_{n-1}} \lim_{r\to\infty} \int_{S_r} \tilde{G}^g(X^{(i)},\nu)\, d\!\vol_{S_r}. 
$$
\end{theo}

\smallskip

{\flushleft\textbf{Acknowledgements}}

The author thanks Piotr Chru\'sciel and Julien Cortier for useful comments, and the referee for his very careful reading of the paper.

\bibliographystyle{smfplain}

\begin{thebibliography}{APS76}

\bibitem{ashtekar-hansen_unified} 
A. Ashtekhar and R.O. Hansen, \emph{A unified treatment of null and spatial infinity in general relativity, {I}.
{U}niversal structure, asymptotic symmetries and conserved quantities at spatial infinity},
J. Math. Phys. \textbf{19} (1978), 1542--1566.

\bibitem{ba}
R. Bartnik, \emph{The mass of an asymptotically flat manifold}, 
Comm. Pure Appl. Math \textbf{39} (1990), 661--693.

\bibitem{bom}
R. Beig and N. \'O Murchadha, \emph{The {P}oincar\'e group as the symmetry group of canonical {G}eneral {R}elativity},
Ann. Phys. \textbf{174} (1987), 463--498.

\bibitem{besse}
A. L. Besse, \emph{Einstein manifolds}, 
Ergeb. Math. Grenzgeb., vol. 10, Springer, Berlin, 1987.

\bibitem{chrusciel-mass}
P. T. Chru{\'s}ciel, \emph{Boundary conditions at spatial infinity from a {H}amiltonian point of view},
\textsl{in} Topological properties and global structure of space-time (P. Bergmann and V. De Sabbata, eds.), 
Plenum Press, New York, NY, 1986, 49-59.

\bibitem{chrusciel_rk-positive-energy} 
P. T. Chru{\'s}ciel, \emph{A remark on the positive energy theorem},
Class. Quant. Grav. \textbf{3} (1986), L115--L121.

\bibitem{ptc-mh}
P. T. Chru{\'s}ciel and M. Herzlich, \emph{The mass of asymptotically hyperbolic Riemannian manifolds},
Pacific J. Math. \textbf{212} (2003), 231--264.

\bibitem{corvino_scalar-curvature}
J. Corvino, \emph{Scalar curvature deformation and a gluing construction for the {E}instein
constraint equations}
Commun. Math. Phys. \textbf{214} (2000), 137--189.

\bibitem{corvino-schoen}
J. Corvino and R. M. Schoen, \emph{On the asymptotics for the vacuum {E}instein constraint equations}
J. Diff. Geom. \textbf{73} (2006), 185--217.

\bibitem{corvino-wu} 
J. Corvino and H. Wu, \emph{On the center of mass of isolated systems},
Class. Quant. Grav. \textbf{25} (2008), 085008.

\bibitem{mh_survey-AH-mass}
M. Herzlich, \emph{Mass formulae for asymptotically hyperbolic manifolds},
\textsl{in} {AdS/CFT} correspondence: {E}instein metrics and their conformal boundaries (O. Biquard, ed.), 
IRMA Lectures in Mathematics and Theoretical Physics, vol. 8, Eur. Math. Soc., Z\"urich, 2005, 103--121.

\bibitem{huang_conf-chinoise} 
L.-H. Huang, \emph{On the center of mass in {G}eneral {R}elativity},
\textsl{in} Fifth International Congress of Chinese mathematicians, Part I, 2, 
AMS/IP Stud. Adv. Math., vol. 51, Amer. Math. Soc., Providence, RI, 2012, 575--591.

\bibitem{miao-tam} 
P. Miao and L.-F. Tam, \emph{Evaluation of the ADM mass and center of mass} via \emph{the Ricci tensor},
Proc. Amer. Math. Soc. \textbf{144} (2016), 753--761.

\bibitem{michel_geometric-invariance}
B. Michel, \emph{Geometric invariance of mass-like invariants}, 
J. Math. Phys. \textbf{52} (2011), 052504.
 
\bibitem{regge-teitelboim_surface-integrals}
T. Regge and C. Teitelboim, \emph{Role of surface integrals in the {H}amiltonian formulation of {G}eneral {R}elativity}, 
Ann. Phys. \textbf{88} (1974), 286--318.

\bibitem{regge-teitelboim_improved}
T. Regge and C. Teitelboim, \emph{Improved {H}amiltonian for {G}eneral {R}elativity}, 
Phys. Lett. B \textbf{53} (1974), 101--105.

\bibitem{schoen_ricci-mass}
R. M. Schoen, \emph{The existence of weak solutions with prescribed singular behaviour for a conformally
invariant scalar equation}
Commun. Pure Appl. Math. \textbf{41} (1988), 317--392.

\bibitem{wang_hyperbolic-mass}
X. Wang, \emph{Mass for asymptotically hyperbolic manifolds}
J. Diff. Geom. \textbf{57} (2001), 273--299.

\end{thebibliography}

\end{document}